\documentclass[12pt]{article}
\usepackage{dsfont}
\usepackage{amsfonts}
\usepackage{amsthm}
\usepackage{amssymb}
\usepackage{mathtools}
\usepackage{amsmath}
\usepackage{mathrsfs}
\usepackage{mathrsfs}
\usepackage{stmaryrd}
\usepackage{framed}
\usepackage{nccmath}
\usepackage{wrapfig}
\usepackage{enumitem}
\usepackage{geometry}
\usepackage{tikz,tikz-cd}
\usepackage{bm}
\usepackage[stable]{footmisc}
\usepackage[new]{old-arrows}
\usepackage{draftwatermark}
\usepackage{quiver}
\usepackage{tikz}
\usepackage{pgf}
\usepackage{pgfplots}
\usepackage{tikz-3dplot}
\usepackage{tkz-base}
\usepackage{tkz-fct}
\usepackage{tkz-euclide}
\usetikzlibrary{math}

\usepackage{amsrefs}
\usepackage{lipsum}

\SetWatermarkText{DRAFT - GSN - 2022}
\SetWatermarkColor[gray]{0.95}
\SetWatermarkScale{0.4}
\NewDocumentCommand{\tens}{t_}
{%
	\IfBooleanTF{#1}
	{\tensop}
	{\otimes}%
}
\NewDocumentCommand{\tensop}{m}
{%
	\mathbin{\mathop{\otimes}\displaylimits_{#1}}%
}

\newcommand{\prtt}[1]{\left( #1 \right)}
\newcommand{\tA}[2]{#1 \tens_{R} #2}

\newcommand{\lips}[2]{\tA{#1}{#2}-\tA{#2}{#1}}

\newcommand{\cL}{{\mathcal L}}

\newcommand{\bC}{{\mathbb{C}}}
\newcommand{\bR}{{\mathbb{R}}}

\newcommand{\bN}{{\mathbb{N}}}
\newcommand{\bZ}{{\mathbb{Z}}}
\newcommand{\sub}{\subseteq}
\newcommand{\ten}{\otimes}
\newcommand{\fp}{\mathfrak{p}}

\newcommand{\bx}{\mathbf{x}}

\newcommand{\ord}{\mbox{ord}}

\newcommand{\spec}{\mbox{Spec}}

\newcommand{\car}{\mbox{char}}

\newcommand{\ideal}[1]{\left\langle #1 \right\rangle}

\newcommand\quotient[2]{{^{\displaystyle #1}}/{_{\displaystyle #2}}}

\theoremstyle{plain}

\newtheorem{definition}{Definition}[section]
\numberwithin{definition}{section}

\newtheorem{proposition}[definition]{Proposition}
\newtheorem{theorem}[definition]{Theorem}
\newtheorem{corollary}[definition]{Corollary}
\newtheorem{lemma}[definition]{Lemma}
\newtheorem{observation}[definition]{Observation}
\newtheorem{example}[definition]{Example}
\newtheorem{remark}[definition]{Remark}

\makeatletter
\def\moverlay{\mathpalette\mov@rlay}
\def\mov@rlay#1#2{\leavevmode\vtop{%
		\baselineskip\z@skip \lineskiplimit-\maxdimen
		\ialign{\hfil$\m@th#1##$\hfil\cr#2\crcr}}}
\newcommand{\charfusion}[3][\mathord]{
	#1{\ifx#1\mathop\vphantom{#2}\fi
		\mathpalette\mov@rlay{#2\cr#3}
	}
	\ifx#1\mathop\expandafter\displaylimits\fi}
\makeatother

\begin{document}
	
	\title{Monomial algebraic and algebroid curves and their Lipschitz saturation over the ground ring}
	\author{Anne Frühbis-Krüger and Thiago da Silva}
	\date{}
	
	\maketitle
	
	\begin{abstract}
		In this article, we provide an explicit description of the Lipschitz saturation $A^*_{B,R}$ of a subalgebra $A\sub R[[t^{\gamma_1},\hdots, t^{\gamma_n}]]$ over a ring $R$ in terms of the numerical semigroup $\Gamma=\ideal{\gamma_1,\hdots,\gamma_n}$ with mild conditions imposed on $R$.
	\end{abstract}

	\tableofcontents

	\section*{Introduction}

	The notion of Lipschitz saturation of algebras was developed by Pham and Teissier in the late 1960s in \cite{PT, P1}. In the case of complex analytic algebras, they observed that the germs of Lipschitz meromorphic functions lie between $A$ and $\overline{A}$, showing that it coincides with the Zariski saturation \cite{Z} in the hypersurface case. Namely, for a reduced complex analytic algebra $A$ with normalization $\overline{A}$, Pham and Teissier defined the Lipschitz saturation of $A$ as $$A^*:=\{f\in\overline{A}\mid f\ten_{\bC}1-1\ten_{\bC}f\in\overline{I_A}\},$$
	
	\noindent where $I_A$ denotes the kernel of the canonical map $\overline{A}\ten_{\bC}\overline{A}\rightarrow \overline{A}\ten_A\overline{A}$. 
	
	More recently, Gaffney used this machinery in \cite{gaffney1} to deal with bi-Lipschitz equisingularity of families of curves and showed that Lipschitz saturation is related to the integral closure of the \textit{double} of the ideal, a concept that he defined in the same article. This was used in  \cite{gaffney2} to get an infinitesimal Lipschitz condition for a family of complex analytic hypersurfaces. In \cite{GS, GS2, Silva1}, the construction of the double of ideals was then generalized to sheaves of modules to obtain an infinitesimal condition for codimensions greater than $1$, obtaining a necessary relationship with Fernandes and Ruas' strongly Lipschitz trivial condition \cite{FR2, FR}. In \cite{SGP, SGP2}, the integral closure of ideals and the double are used to describe the bi-Lipschitz equisingularity of families of matrices.
	
	In \cite{L}, Lipman extended the definition of Lipschitz saturation for a sequence of ring morphisms $R\rightarrow A\rightarrow B$ and defined what he called the \textit{relative Lipschitz saturation of $A$ in $B$}, denoted by $A^*_{B, R}$. Besides, in \cite{SR}, the authors extended the results of Lipman concerning the property of contraction for the integral and universally injective cases, and in \cite{bernard}, the author studies the relative Lipschitz saturation of complex algebraic varieties.
	
	Although the concept of Lipschitz saturation of algebras has been explored since the 1970s, effectively calculating the Lipschitz saturation of a given algebra remains a significant challenge. Suppose that $A=\bC[C]$ is the coordinate ring of a monomial irreducible complex analytic curve $C \subseteq \mathbb{C}^n$ in its normalization algebra $B=\overline{\bC[C]}=\bC\{t\}$, where the normalization map is induced by an analytic map $\phi:\bC\rightarrow C$ on the form $$\phi(t)=(t^{\gamma_1}, t^{\gamma_2}\cdots,\hdots, t^{\gamma_n}).$$
	
	For giving an explicit description of the Lipschitz saturation $A^*_{B,\bC}$, the semigroup $\Gamma=\langle\gamma_1, \hdots,\gamma_n\rangle$ of $A=\bC[C]=\bC\{t^{\gamma_1}, \hdots, t^{\gamma_n}\}$ will play a fundamental role.
	
	With this motivation in mind, in this work, we will describe the Lipschitz saturation of a $R$-subalgebra $A$ of $B$ in a more general scenario. More precisely, we work with the following setup:
	
	\begin{itemize}
		\item $R$ is a ring with some additional properties (see Theorem \ref{202407191955});
		
		\item $B$ is an $R$-subalgebra of the power series ring $R[[t]]$ containing the polynomial ring $R[t]$;
		
		\item $A$ is an $R$-subalgebra of $B\cap R[[t^{\gamma_1},\hdots,t^{\gamma_n}]]$ containing $R[t^{\gamma_1},\hdots,t^{\gamma_n}]$.
		
	\end{itemize}
	
	In Section 1, we recall some basic definitions and results about the relative Lipschitz saturation of algebras obtained by Lipman in \cite{L}. In Section 2, we develop some basic results on the Lipschitz saturation of $R[t^{\gamma_1},\hdots, t^{\gamma_n}]$ in $R[t]$, where some of the suitable properties for the ring $R$ will arise. Thereafter, we prove Theorem \ref{202407191955} in Section 3, which is the main result of this work, describing $A^*_{B,R}$ completely in terms of its semigroup. In the appendix, we present explicit applications of Theorem \ref{202407191955} in the particular case where $R$ is a domain.

	\section{Background and some results on relative Lipschitz saturation of algebras}
	
	In this section, we will recall the definition of relative Lipschitz saturation given by Lipman in \cite{L} as well as some results that will be necessary throughout this work.
	
	Let $R$ be a ring\footnote{In this work, all rings are commutative and unitary.}, let $A, B$ be $R$-algebras and consider a composition of algebra morphisms 
	\begin{align*}
		R \overset{\tau}{\longrightarrow} A \overset{g}{\longrightarrow} B.
	\end{align*} Consider the map from $B$ to its tensor product by $R$, in a diagonal way:
	\begin{align*}
		\Delta : B &\longrightarrow \tA{B}{B} \\
		b &\longmapsto \lips{b}{1}.
	\end{align*}
	
	It is easy to conclude the following properties of $\Delta$:
	\begin{enumerate}
		\item $\Delta(b_1+b_2)=\Delta(b_1+b_2), \forall b_1,b_2\in B$;
		
		\item $\Delta(rb)=r\Delta(b), \forall r\in R$ and $b\in B$;
		
		\item (Leibniz rule) $\Delta(b_1b_2)=(b_1\ten_R 1)\Delta(b_2)+(1\ten_Rb_2)\Delta(b_1), \forall b_1,b_2\in B$. 
	\end{enumerate}
	In particular, $\Delta$ is an $R$-module morphism, but not an $R$-algebra morphism.
	
	By the universal property of the tensor product, there exists a unique $R$-algebra morphism
	$$\varphi: \tA{B}{B} \longrightarrow B\tens_{A}B$$
	which maps $x\otimes_R y\mapsto x\otimes_A y$, for all $x,y\in B$. Furthermore, we already know (see \cite{kleiman}, 8.7, for the first equality) that
	\begin{eqnarray*}
		\ker\varphi & = &
		\ideal{ \{g(a)x\ten_R y-x\ten_Rg(a)y\mid a\in A\mbox{ and }x,y\in B\}}\\
		& = &
		\ideal{\{ g(a)\ten_R1-1\ten_Rg(a)\mid a\in A\}}\\
		& = & \Delta(g(A))(B\ten_RB).
	\end{eqnarray*}
	i.e, $\ker\varphi$ is the ideal of $B\ten_RB$ generated by the image of $\Delta\circ g$.
	
	Motivated by the works of Pham and Teissier \cite{PT,P1}, Lipman introduced the following definition in \cite{L}.
	
	\begin{definition}[\cite{L}]
		The Lipschitz saturation of $A$ in $B$ relative to $R \overset{\tau}{\rightarrow} A \overset{g}{\rightarrow} B$ is the set
		\begin{align*}
			A^*_{B,R} := \left\{x \in B \mid \Delta(x) \in \overline{\ker \varphi}\right\}.
		\end{align*}
	\end{definition}
	If $A^*_{B,R} = g(A)$ then $A$ is said to be Lipschitz saturated in $B$. An important particular case is when $A$ is an $R$-subalgebra of $B$, taking $g$ as the inclusion. For this case, we recall some properties proved by Lipman in \cite{L}, which in summary shows that the relative Lipschitz saturation in this setting does not show any unexpected behaviour.
	
	\begin{proposition}[\cite{L}]\label{prop_satsubanel}
		$A^*_{B,R}$ is an $R$-subalgebra of $B$ which contains $g(A)$.
	\end{proposition}
	
	\begin{proposition}[\cite{L}]\label{202309051604}
		Let $A$ be an $R$-subalgebra of $B$ and let $C$ be an $R$-subalgebra of $B$ containing $A$ as an $R$-subalgebra. \[\begin{tikzcd}
			R & A & C & B
			\arrow["\tau", from=1-1, to=1-2]
			\arrow[hook, from=1-2, to=1-3]
			\arrow[hook, from=1-3, to=1-4]
			\arrow["g"', curve={height=12pt}, hook, from=1-2, to=1-4]
		\end{tikzcd}\]

		Then $A^*_{B,R} \subseteq C^*_{B.R}$.
	\end{proposition}

	\begin{corollary}[\cite{L}]
		If $A$ is an $R$-subalgebra of $B$ then: 
		
		\begin{enumerate}
			\item $A \subseteq A^*_{B,R}$.
			
			\item $(A^*_{B,R})^*_{B,R} = A^*_{B,R}$.
		\end{enumerate}
	\end{corollary}
	
	\begin{proposition}[\cite{L}]\label{202405091748}
		Suppose that
		\[\begin{tikzcd}
			R & A & B \\
			{R'} & {A'} & {B'}
			\arrow["\tau", from=1-1, to=1-2]
			\arrow["g", from=1-2, to=1-3]
			\arrow["{\tau'}", from=2-1, to=2-2]
			\arrow["{g'}", from=2-2, to=2-3]
			\arrow["{f_R}"', from=1-1, to=2-1]
			\arrow["{f}", from=1-3, to=2-3]
			\arrow["{f_A}", from=1-2, to=2-2]
		\end{tikzcd}\]
		
		\noindent is a commutative diagram of algebra morphisms. Then
		$$f\prtt{A^*_{B,R}} \subseteq \prtt{A'}^*_{B',R'}.$$
	\end{proposition}

	\section{Monomials in the Lipschitz saturation of $R[t^{\gamma_1},\hdots,t^{\gamma_n}]$ in $R[t]$}
	
	We have already seen in the previous section that a key to understanding Lipschitz saturations is the study of $\overline{\ker\varphi}$. Through a sequence of technical lemmata, we obtain an explicit description thereof for the case of a monomial $R$-algebra, which only relies on the semigroup. Within this section, however, we only prove one inclusion and leave the other one to the following section. We start by a rather trivial, but useful observation:
	
	\begin{observation}\label{2024051222082}
		Let $R$ be a ring and $d,\alpha,s\in\bN$ with $s\geq 2$. If $a,b\in R$ then $$a^{\alpha+sd}-b^{\alpha+sd}\in\ideal{a^{\alpha+(s-2)d}-b^{\alpha+(s-2)d}, a^{\alpha+(s-1)d}-b^{\alpha+(s-1)d}}.$$
	\end{observation}
	
	\begin{proof}
		Direct calculation provides $$(a^d+b^d)(a^{\alpha+(s-1)d}-b^{\alpha+(s-1)d})=a^{\alpha+sd}-a^db^{\alpha+(s-1)d}+a^{\alpha+(s-1)d}b^d-b^{\alpha+sd}$$$$=(a^{\alpha+sd}-b^{\alpha+sd})+a^db^d(-b^{\alpha+(s-2)d}+a^{\alpha+(s-2)d}).$$ Hence $a^{\alpha+sd}-b^{\alpha+sd}=(-a^db^d)(a^{\alpha+(s-2)d}-b^{\alpha+(s-2)d})+(a^d+b^d)(a^{\alpha+(s-1)d}-b^{\alpha+(s-1)d}).$
	\end{proof}
	
	To make the following statements and proofs more readable, let us fix a useful shorthand notation: for each $x\in B$, denote $x_1:=x\ten_R 1_B$ and $x_2:=1_B\ten_Rx$. Thus, $\Delta(x)=x_1-x_2$ and for all $r\in\bN$ we have
	$$\Delta(x^r)=x^r\ten_R1_B-1_B\ten_Rx^r=(x\ten_R1_B)^r-(1_B\ten_Rx)^r=x_1^r-x_2^r.$$

	\begin{proposition}\label{2024051512221}
		Let $R\rightarrow A\overset{g}{\rightarrow}B$ be a composition of ring morphisms. If $x\in B$ and $\alpha,d\in\bN$ is such that $x^\alpha,x^{\alpha+d}\in A^*_{B,R}$ then $x^{\alpha+sd}\in A^*_{B,R}$, for all $s\in\bN$. In particular, if $x^\alpha,x^{\alpha+1}\in A^*_{B,R}$ then $x^m\in A^*_{B,R}$, for all $m\in\bN$ with $m\geq \alpha$.
	\end{proposition}
	
	\begin{proof}
		We shall prove the claim by induction on $s$. The case $s=1$ is true by hypothesis. Suppose that $s>1$ and $x^{\alpha+rd}\in A^*_{B,R}$, for all $r\in\{1,\hdots,s-1\}$. In particular, $$x^{\alpha+(s-2)d},x^{\alpha+(s-1)d}\in A^*_{B,R},$$\noindent  which implies that     $x_1^{\alpha+(s-2)d}-x_2^{\alpha+(s-2)d},x_1^{\alpha+(s-1)d}-x_2^{\alpha+(s-1)d}\in\overline{\ker\varphi}$, and consequently
		$$\ideal{x_1^{\alpha+(s-2)d}-x_2^{\alpha+(s-2)d},x_1^{\alpha+(s-1)d}-x_2^{\alpha+(s-1)d}}\sub\overline{\ker\varphi}.$$
		By Observation \ref{2024051222082} we have $x_1^{\alpha+sd}-x_2^{\alpha+sd}\in\ideal{x_1^{\alpha+(s-2)d}-x_2^{\alpha+(s-2)d}, x_1^{\alpha+(s-1)d}-x_2^{\alpha+(s-1)d}}$, which implies that $x_1^{\alpha+sd}-x_2^{\alpha+sd}\in\overline{\ker\varphi}$, i.e, $x^{\alpha+sd}\in A^*_{B,R}$.
	\end{proof}
	
	\begin{proposition}\label{202406031700}
		Let $R\rightarrow A\overset{g}{\rightarrow}B$ be a composition of ring morphisms and $\bx_1,\hdots,\bx_s\in B$. Then, for all $i_1,\hdots,i_s\in\bZ_{\geq0}$, one has
		$$\Delta(\bx_1^{i_1}\cdots\bx_s^{i_s})\in\ideal{\Delta(\bx_1),\hdots,\Delta(\bx_s)}.$$	
		In particular, $\Delta(R[\bx_1,\hdots,\bx_s])\in \ideal{\Delta(\bx_1),\hdots,\Delta(\bx_s)}$.
	\end{proposition}
	
	\begin{proof}
		For the case $s=1$, we proceed by induction on $i_1$, and using the Leibniz rule to get the equation $$\Delta(\bx_1^{i_1})=(\bx_1^{i_1-1}\ten_R 1)\Delta(\bx_1)+(1\ten_R\bx_1)\Delta(\bx_1^{i_1-1}),$$
		which completely solves this case.               Now suppose that $s>1$ and the result is true for $s-1$. Using Leibniz rule again, one has
		$$\Delta(\bx_1^{i_1}\cdots \bx_{s}^{i_{s}})=((\bx_1^{i_1}\cdots \bx_{s-1}^{i_{s-1}})\ten_R 1)\underbrace{\Delta(\bx_s^{i_s})}_{\in\ideal{\Delta(\bx_1),\hdots,\Delta(\bx_s)}}+(1\ten_R\bx_s^{i_s})\underbrace{\Delta(\bx_1^{i_1}\cdots \bx_{s-1}^{i_{s-1}})}_{\in\ideal{\Delta(\bx_1),\hdots,\Delta(\bx_s)}}.$$
	\end{proof}
	
	\begin{corollary}\label{2024050712061}
		Let $B$ be an $R$-algebra, $\bx_1,\hdots,\bx_n\in B$ and $A=R[\bx_1,\hdots,\bx_n]$. If $\varphi:B\ten_RB\rightarrow B\ten_AB$ is the canonical map and $\Delta:B\rightarrow B\ten_RB$ as before then $$\ker\varphi=\ideal{\Delta(\bx_1),\hdots,\Delta(\bx_n)}.$$
		
	\end{corollary} 
	
	\begin{proof}

		We already know that $\ker\varphi=\Delta(g(A))(B\ten_RB)=\Delta(A)(B\ten_RB)$. Since $A=R[\bx_1,\hdots,\bx_n]$ then the result now is a consequence of Proposition \ref{202406031700}.
	\end{proof}	
	
	Using the above observations and results, we are now ready to state and prove one of the central tools of this article.
	
	\begin{proposition}\label{202407192004}
		Let $T$ be a noetherian ring, $m\in\bN$ with $m\geq 2$, $\alpha_1<\hdots<\alpha_m$ natural numbers, $d = \gcd(\alpha_1,\dots,\alpha_m)$ and $z_1,z_2\in  T$. Then  
		$$z_1^{\alpha_m+d}-z_2^{\alpha_m+d}\in\overline{\ideal{z_1^{\alpha_1}-z_2^{\alpha_1},\hdots,z_1^{\alpha_m}-z_2^{\alpha_m}}}.$$
		
		In particular, $z_1^{\alpha_m+sd}-z_2^{\alpha_m+sd}\in\overline{\ideal{z_1^{\alpha_1}-z_2^{\alpha_1},\hdots,z_1^{\alpha_m}-z_2^{\alpha_m}}}, \forall s\in\bN$.
	\end{proposition}
	
	\begin{proof}
		Let $h:=z_1^{\alpha_m+d}-z_2^{\alpha_m+d}$ and $I:=\ideal{z_1^{\alpha_1}-z_2^{\alpha_1},\hdots,z_1^{\alpha_m}-z_2^{\alpha_m}}$. By the Valuative Criterion \cite[Theorem 6.8.3]{SH} it suffices to prove that $\psi(h)\in IV:=\psi(I)V$ for all discrete valuation rings $V$ and all ring morphisms $\psi:T\rightarrow V$. Let $V$ be a discrete valuation ring with field of fractions $K$ and let $\psi:T\rightarrow V$ be a ring morphism. If $\psi(h)=0$ then $\psi(h)\in IV$. Now suppose $\psi(h)\neq 0$.
		
		Let $t\in V$ be a uniformizing parameter of $V$ and let $\ord_V$ be the discrete valuation map associated to $V$ on $K$. Since $x_1:=\psi(z_1), x_2:=\psi(z_2)\in V$ then we can write $x_1=u_1t^r$ and $x_2=u_2t^s$, with $u_1,u_2\in V^\times$, $\ord_V(x_1)=r$ and $\ord_V(x_2)=s$. Thus, $\ord_V(u_1)=\ord_V(u_2)=0$, $\psi(h)=x_1^{\alpha_m+d}-x_2^{\alpha_m+d}=u_1^{\alpha_m+d}t^{r(\alpha_m+d)}-u_2^{\alpha_m+d}t^{s(\alpha_m+d)}$ and $$IV=\ideal{u_1^{\alpha_1}t^{r\alpha_1}-u_2^{\alpha_1}t^{s\alpha_1},\hdots,u_1^{\alpha_m}t^{r\alpha_m}-u_2^{\alpha_m}t^{s\alpha_m}}.$$
		
		We first treat the case that $r\neq s$. Taking $v:=\min\{r,s\}$, one has $\alpha_j v=\ord_V(x_1^{\alpha_j}-x_2^{\alpha_j})$ for all $j\in\{1,\hdots,m\}$. Thus, $IV=\ideal{t^{\alpha_1v}}$ and since $\ord_V(\psi(h))=(\alpha_m+\gamma) v > \alpha_1v$, we see that $\psi(h)\in IV$. 
		
		In the remaining case $r=s$ we observe that $u_1,u_2\in V^\times$ allows us to rewrite $u_1^\ell-u_2^\ell=u_1^\ell(1_V-u^\ell)$ for all $\ell\in\bN.$ where $u:=\frac{u_2}{u_1}\in V^\times$. Then
		
		\begin{eqnarray*}
			IV & = & \ideal{u_1^{\alpha_1}(1_V - u^{\alpha_1})t^{r\alpha_1},\hdots,u_1^{\alpha_m}(1_V-u{\alpha_m})t^{r\alpha_m}} \cr
			& = & \ideal{(1_V - u^{\alpha_1})t^{r\alpha_1},\hdots,(1_V-u^{\alpha_m})t^{r\alpha_m}} \cr
			\psi(h) & = & u_1^{\alpha_m+d}(1_V-u^{\alpha_m+d})t^{r(\alpha_m+d)}
			= u_1^{\alpha_m+d}(1_V - u^d)\left( \sum_{i=0}^{\frac{\alpha_m}{d}} (u^d)^i \right) t^{r(\alpha_m+d)}
		\end{eqnarray*}
		
		Obviously, $(1_V - u^{\alpha_i})t^{r\alpha_m} \in IV$ for all $1 \leq i \leq m$. As $(1_V - u^\ell) - u^{\ell - k}(1_V - u^k) = 1_V - u^{\ell -k}$ for any $k, \ell  \in {\mathbb N}$, we may mimic the steps of the euclidean algorithm for the $\alpha_i$ in the exponents. This implies that $(1_V -u^{\gcd(\alpha_1,\hdots,\alpha_m)})t^{r\alpha_m} \in IV$, which in turn is clearly a divisor of $\psi(h)$.
		
		For the last statement, it suffices to use induction on $s$ as we have done in Proposition \ref{2024051512221}.\end{proof}
	
	We denote $R[t]$ as the polynomial ring in one indeterminate over $R$.
	
	\begin{corollary}\label{202408051848}
		Suppose that $R$ is a noetherian ring. Let $m\in\bN$ with $m\geq 2$, and let $0<\alpha_1<\cdots<\alpha_m$ be natural numbers with $d=\gcd(\alpha_1,\hdots,\alpha_m)$. Then, $$t^{\alpha_m+sd}\in R[t^{\alpha_1},\hdots, t^{\alpha_m}]^*_{R[t],R}, \forall s\in\bN.$$
	\end{corollary}
	
	\begin{proof}
		It is a consequence of Corollary \ref{2024050712061} by applying Proposition \ref{202407192004} on the noetherian ring $T:=B\ten_RB\cong R[t_1,t_2]$.
	\end{proof}
	
	\section{Lipschitz saturation of monomial subalgebras of $R[[t]]$}
	
	Let $n\in\bN$, $n\geq 2$ and let  $\gamma_1<\hdots<\gamma_n$ be coprime natural numbers, with $\gamma_i\nmid\gamma_j, \forall i,j\in\{1,\hdots,n\}$ with $i<j$.\footnote{For the purposes of this work, this last condition is not restrictive because we can always reduce to this case.} Denote $d_j:=\gcd(\gamma_1,\hdots,\gamma_j),\forall j\in\{1,\hdots,n\}$. We will see these partial common divisors will play a fundamental role in the description of the Lipschitz saturation in this context. 
	
	Let $B$ be an $R$-subalgebra of $R[[t]]$ and suppose that $A$ is an $R$-subalgebra of $R[[t^{\gamma_1},\hdots,t^{\gamma_n}]]\cap B$ such that: 
	
	\begin{itemize}
		\item $R[t]$ is a $R$-subalgebra of $B$;
		
		\item $R[ t^{\gamma_1},\hdots, t^{\gamma_n}]$ is a $R$-subalgebra of $A$.
	\end{itemize} 
	
	We have already seen that $A^*_{B,R}=\Delta^{-1}(\overline{\ker\varphi}),$ where $\Delta: B\rightarrow B\ten_RB$ is the diagonal morphism and $\varphi:B\ten_RB\rightarrow B\ten_AB$ is the canonical morphism. 
	
	For a field $k$, as usual, we denote by $k[[X_1,\dots,X_r]]$ the ${\mathfrak m}$-adic completion of $k[X_1,\dots,X_r]$, where ${\mathfrak m} = \langle X_1,\dots,X_r\rangle$ . The order of an element $f \in k[[X_1,\dots,X_r]] \setminus \{0\}$ is the least $\ell \in {\mathbb Z}_{\geq 0}$ such that the image of $f$ under the projection to $k[X_1,\dots,X_r]/{\mathfrak m}^{\ell +1} \cong k[[X_1,\dots,X_r]]/{\mathfrak m}^{\ell +1}$ does not vanish.
	
	Note that the flatness of $k$-algebras provides the following sequence of canonical inclusions for any $k$-subalgebra $C  \subseteq k[[t]]$:
	$$C \ten_k C \rightarrow k[[t]] \ten_k k[[t]] \rightarrow k[[t]] \widehat{\otimes}_k k[[t]] \cong k[[t_1,t_2]],$$
	where $\widehat{\otimes}_k$ is the completed tensor product. As a consequence we have the following observation, which we state for further references.
	
	\begin{observation}\label{202411261738}
		Let $C$ be a $k$-subalgebra of $k[[t]]$, let	$s\in\bN$ and $\bx_1,\hdots,\bx_s\in\ideal{t}\cap C$ and suppose that $A'$ is a $k$-subalgebra of $k[[\bx_1,\hdots,\bx_s]]\cap C$. If $\varphi':C\ten_kC\rightarrow C\ten_{A'}C$ is the canonical morphism then
		
		$$\ker\varphi'\sub \ideal{\Delta(\bx_1),\hdots,\Delta(\bx_s)}_{k[[t_1,t_2]]}.\footnote{For a ring $R$ and $a_1,\hdots,a_m\in R$, to emphasize in which ring we are working, we use the notation $\ideal{a_1,\hdots,a_m}_R$ for the ideal of $R$ generated by $a_1,\hdots,a_m$.}$$
	\end{observation}
	
	We denote the numerical semigroup of ${\mathbb Z}_{\geq 0}$ generated by $\gamma_1,\dots,\gamma_n$ by $\Gamma$ and the set of gaps ${\mathbb Z}_{\geq 0} \setminus \Gamma$ by $G(\Gamma)$. As we require that $\gcd(\gamma_1,\hdots,\gamma_n) = 1$, the set $G(\Gamma)$ is finite. The maximal element of $G(\Gamma)$ will be referred to as $g(\Gamma)$ and the conductor of $\Gamma$ is $c(\Gamma) = g(\Gamma) +1$. 
	
	\begin{observation}\label{202405081704}
		Let $\ell\in\bZ_{\geq 0}$. Then: 
		$$t^{\ell}\in A\iff \ell \in \Gamma$$ 
		and, in particular, $t^{\ell} \in A, \forall \ell \geq c(\Gamma)$.
	\end{observation}
	
	For describing  $A^*_{B,R}$ explicitly, we introduce some shorthand notation:
	\begin{itemize}
		\item For each $Y\sub \bZ_{\geq 0}$ we denote $t^Y:=\{t^{\ell}\mid \ell\in Y\}$;
		
		\item $L_j:=\{\ell\in G(\Gamma)\mid \gamma_j<\ell<\gamma_{j+1}\mbox{ and }d_j\mid\ell\}$, $\forall j\in\{1,\hdots,n-1\}$\footnote{Since $d_1=\gamma_1$ then $L_1=\emptyset$.};
		
		\item $\tilde{L}(m):=\{\ell\in G(\Gamma)\mid \gamma_m+1\leq\ell\leq \gamma_m+\gamma_1-1\}$, $\forall m\in\{2,\hdots,n\}$;
		
		\item $L(m):=\left(\bigcup\limits_{j=1}^{m-1}L_j\right)\cup\tilde{L}(m)$.
	\end{itemize}
	
	From Corollary \ref{202408051848} we can directly deduce the following statement:
	
	\begin{proposition}\label{202408151302}
		
		Suppose that $R$ is a noetherian ring. If $r\in\{2,\hdots,n\}$ and $d_r=1$ then  $$A[t^{L(r)}]\sub A^*_{B,R}.$$
	\end{proposition}
	
	\begin{proof}
		It suffices to show that $t^{L(r)}\sub A^*_{B,R}$. For $j\in\{1,\hdots,r-1\}$ and $\ell\in L_j$, we know that $d_j\mid \gamma_j$ and $d_j\mid \ell$, whence there exists $s\in\bN$ such that $\ell=\gamma_j+sd_j$. By Corollary \ref{202408051848} we have $$t^\ell\in R[t^{\gamma_1},\hdots, t^{\gamma_j}]^*_{R[t],R}\sub A^*_{B,R}.$$
		
		As, in particular, $d_r=1$, Corollary \ref{202408051848} also implies that $t^{\gamma_r+s}\in R[t^{\gamma_1},\hdots, t^{\gamma_r}]^*_{R[t],R}\sub A^*_{B,R}$ for all $s \in {\mathbb Z}_{\geq 0}$. Since $t^{\gamma_j}\in A\sub A^*_{B,R}$, Proposition \ref{2024051512221} ensures that $t^{\tilde{L}(r)}\sub A^*_{B,R}$.
	\end{proof}
	
	Is this already all of $A^*_{B,R}$ or are there still contributions missing? The following proposition and its corollaries provide the central argument why we have already reached the desired explicit description.
	
	For any $m\in\bN$ and any field $E$, fix the notation 
	$$\mu_E(m):=\{\zeta\in E\mid \zeta^m=1\}$$
	for the multiplicative group of the $m^{\mbox{\tiny{th}}}$-roots of the unity in $E$.
	
	\begin{proposition}\label{202408082116}
		Let $j\in\{1,\hdots,n-1\}$ and $f=\sum\limits_{i\in\bZ_{\geq0}} a_it^i\in A^*_{B,R}$, with $a_i\in R, \forall i\in\bZ_{\geq 0}$. Suppose that $\fp\in\spec R$ satisfies $\car\left(\quotient{R}{\fp}\right)\nmid d_j$. Then $$a_\ell\in\fp,\forall \ell\in\{1,\hdots,\gamma_{j+1}-1\}\mbox{ with  }d_j\nmid\ell.$$
	\end{proposition}
	
	\begin{proof}
		Let $\kappa$ be the field of fractions of the domain $\quotient{R}{\fp}$, and let $\overline{\kappa}$ be the algebraic closure of $\kappa$. We abreviate $A':=\overline{\kappa}[[t^{\gamma_1},\hdots, t^{\gamma_n}]]$ and $B':=\overline{\kappa}[[t]]$. Using the canonical surjection to a quotient and the inclusion of a domain in its field of fractions on the coefficients, we have a canonical commutative diagram % https://q.uiver.app/#q=WzAsNixbMCwwLCJrIl0sWzEsMCwiQSJdLFsyLDAsIkIiXSxbMCwxLCJcXG92ZXJsaW5le2t9Il0sWzEsMSwiQSciXSxbMiwxLCJCJyJdLFswLDMsIiIsMCx7InN0eWxlIjp7InRhaWwiOnsibmFtZSI6Imhvb2siLCJzaWRlIjoidG9wIn19fV0sWzEsNCwiIiwwLHsic3R5bGUiOnsidGFpbCI6eyJuYW1lIjoiaG9vayIsInNpZGUiOiJ0b3AifX19XSxbMiw1LCJcXHBzaSIsMCx7InN0eWxlIjp7InRhaWwiOnsibmFtZSI6Imhvb2siLCJzaWRlIjoidG9wIn19fV0sWzAsMSwiIiwxLHsic3R5bGUiOnsidGFpbCI6eyJuYW1lIjoiaG9vayIsInNpZGUiOiJ0b3AifX19XSxbMSwyLCIiLDEseyJzdHlsZSI6eyJ0YWlsIjp7Im5hbWUiOiJob29rIiwic2lkZSI6InRvcCJ9fX1dLFszLDQsIiIsMSx7InN0eWxlIjp7InRhaWwiOnsibmFtZSI6Imhvb2siLCJzaWRlIjoidG9wIn19fV0sWzQsNSwiIiwxLHsic3R5bGUiOnsidGFpbCI6eyJuYW1lIjoiaG9vayIsInNpZGUiOiJ0b3AifX19XV0=
		\[\begin{tikzcd}
			R & A & B \\
			{\overline{\kappa}} & {A'} & {B'}
			\arrow[hook, from=1-1, to=1-2]
			\arrow[from=1-1, to=2-1]
			\arrow[hook, from=1-2, to=1-3]
			\arrow[from=1-2, to=2-2]
			\arrow["\Psi", from=1-3, to=2-3]
			\arrow[hook, from=2-1, to=2-2]
			\arrow[hook, from=2-2, to=2-3]
		\end{tikzcd}\]
		and $\sum\limits_{i\in\bZ_{\geq0}}\overline{a_i}t^i=\Psi(f)\in (A')^*_{B',\overline{\kappa}}$ (where $\overline{x}=x+\fp,\forall x\in R$). 
		
		For $B'$ and $A'$, let $\Delta':B'\rightarrow B'\ten_{\overline{\kappa}}B'$ be the diagonal map, and $\varphi':B'\ten_{\overline{\kappa}}B'\rightarrow B'\ten_{A'}B'$ the canonical morphism. Then $\sum\limits_{i\in\bZ_{\geq0}}\overline{a_i}(t_1^i-t_2^i)=\Delta'(\Psi(f))\in\overline{\ker\varphi'}$. Further, in Observation \ref{202411261738} we saw that $$\ker\varphi'\sub\ideal{t_1^{\gamma_1}-t_2^{\gamma_1},\hdots,t_1^{\gamma_n}-t_2^{\gamma_n}}_{\overline{\kappa}[[t_1,t_2]]}.$$
		
		As $\mu_{\overline{\kappa}}(d_j)$ is known to be a cyclic subgroup of the multiplicative group $\overline{\kappa}^*$, there exists $\xi\in \overline{\kappa}$ such that $\mu_{\overline{\kappa}}(d_j)=\langle\xi\rangle$. Since $\car(\overline{\kappa})\nmid d_j$,  $X^{d_j}-1\in \overline{\kappa}[X]$ is separable. Because  $\overline{\kappa}$ is algebraically closed, $|\mu_{\overline{\kappa}}(d_j)|=d_j$. Besides, by definition we know that $d_j$ is a common divisor of $\gamma_1,\hdots,\gamma_j$, and this implies $\xi^{\gamma_s}=1,\forall s\in\{1,\hdots,j\}$. We use again the Valuative Criterion, which we already used in the proof of Proposition \ref{202407192004}, but this time with the knowledge that we are dealing with an element of the integral closure. To this end, we choose a particular ring morphism
		$$\begin{matrix}
			\Phi: & \overline{\kappa}[[t_1,t_2]] & \longrightarrow & \overline{\kappa}[[t]]\\
			&    F       & \longmapsto     & F(t,\xi t)
		\end{matrix}$$
		\noindent which takes $t_1\mapsto t$ and $t_2\mapsto \xi t$. The ideal generated by the image of $\ker\varphi'$ in $\overline{\kappa}[[t]]$ satisfies  
		\begin{eqnarray*}
			\Phi(\ker\varphi')\overline{\kappa}[[t]] & \sub & \ideal{\Phi(t_1^{\gamma_1}-t_2^{\gamma_1}), ,\hdots,\Phi(t_1^{\gamma_n}-t_2^{\gamma_n})} \cr
			& = &\ideal{\underbrace{(1-\xi^{\gamma_1})}_{=0}t^{\gamma_1},\hdots,\underbrace{(1-\xi^{\gamma_j})}_{=0}t^{\gamma_j}, (1-\xi^{\gamma_{j+1}})t^{\gamma_{j+1}},\hdots,(1-\xi^{\gamma_n})t^{\gamma_n}} \cr
			& \sub & \ideal{t^{\gamma_{j+1}}}.
		\end{eqnarray*}
		
		Since $\sum\limits_{i\in\bZ_{\geq0}}\overline{a_i}(t_1^i-t_2^i)\in\overline{\ker\varphi'}$, the persistence of the integral closure (see \cite[Remark 1.1.3 (7)]{SH}) implies that $$\Phi\left(\sum\limits_{i\in\bZ_{\geq0}}\overline{a_i}(t_1^i-t_2^i)\right)\in \overline{\Phi(\ker\varphi')\overline{\kappa}[[t]]}\sub\overline{\ideal{t^{\gamma_{j+1}}}}=\ideal{t^{\gamma_{j+1}}}.$$
		
		\noindent Thus, $\sum\limits_{i\geq 1}\overline{a_i}(1-\xi^i)t^i\in \ideal{t^{\gamma_{j+1}}}$, and consequently $\overline{a_i}(1-\xi^i)=0_\kappa, \forall i\in\{1,\hdots,\gamma_{j+1}-1\}$. 
		
		If $a_i\notin\fp$ , i.e. $\overline{a_i}\neq 0_\kappa$, the above equation implies that $\xi^i=1$. Thus, $|\langle \xi\rangle|$ divides $i$ by Lagrange's Subgroup Theorem, i.e., $d_j\mid i$. Therefore, the contrapositive says that $$a_i\in\fp, \forall i\in\{1,\hdots,\gamma_{j+1}-1\} \mbox{ such that } d_j\nmid i.$$
	\end{proof}
	
	\begin{corollary}\label{202405091127}
		
		Suppose that $R$ is a domain, let $j\in\{1,\hdots,n-1\}$ and $f=\sum\limits_{i\in\bZ_{\geq0}} a_it^i\in A^*_{B,R}$, with $a_i\in R, \forall i\in\bZ_{\geq 0}$. Suppose that $\car(R)\nmid d_j$. If $\ell\in\{1,\hdots,\gamma_{j+1}-1\}$ and $d_j\nmid\ell$ then $a_{\ell}=0$.
	\end{corollary}
	
	\begin{proof}
		Take $\fp:=\ideal{0_R}$.
	\end{proof}

	\begin{corollary}\label{202408082140}
		Suppose that $R$ is a reduced ring, let $j\in\{1,\hdots,n-1\}$ and $f=\sum\limits_{i\in\bZ_{\geq0}} a_it^i\in A^*_{B,R}$, with $a_i\in R, \forall i\in\bZ_{\geq 0}$. Suppose that $\car\left(\quotient{R}{\fp}\right)\nmid d_j, \forall \fp\in\spec R$. If $\ell\in\{1,\hdots,\gamma_{j+1}-1\}$ and $d_j\nmid\ell$ then $a_{\ell}=0$.
	\end{corollary}
	
	\begin{proof}
		Let $\ell\in\{1,\hdots,\gamma_{j+1}-1\}$ with $d_j\nmid\ell$. By Proposition \ref{202408082116} we have $$a_\ell\in\bigcap\limits_{\fp\in\mbox{\scriptsize{Spec}} R}\fp=\sqrt{\ideal{0_R}},$$
		
		\noindent and since $R$ is reduced then $a_{\ell}=0$.
	\end{proof}
	In the case that $R$ is reduced and artinian\footnote{It is known that reduced and artinian is equivalent to be a finite direct product of fields} we know that $\spec R$ is finite. Thus, the previous corollary is always applicable if we choose $\gamma_1$ not divisible by any of the numbers of the finite set $\left\{\car\left(\quotient{R}{\fp}\right)\mid \fp\in \spec R\right\}$.
	
	In Observation \ref{202405081704} we have seen that $t^i\in A, \forall i\in \Gamma$. For any $f\in B$, we can use $B\sub R[[t]]$ and write $f=\sum\limits_{i\in\bZ_{\geq 0}} a_it^i, a_i\in R$. We decompose $f$ into the following two sums:
	$$f = f_{\mid_{G(\Gamma)}} + (f - f_{\mid_{G(\Gamma)}}) 
	:=\sum\limits_{i\in G(\Gamma)}a_it^i + \sum\limits_{i \in \Gamma} a_it^i.$$
	
	\begin{definition}
		We say that $A$ is $(\Gamma, B)$-closed if 
		$$f-f_{\mid_{G(\Gamma)}}\in A$$
		for all $f\in B$.
	\end{definition}
	
	The main cases of interest in the context of Lipschitz saturation turn out to possess this property:
	
	\begin{proposition}\label{202411260046}
		\begin{enumerate}
			\item [a)] If $B=R[t]$ then $A$ is $(\Gamma, B)$-closed;
			
			\item [b)] If $B=R[[t]]$ and $A=R[[t^{\gamma_1},\hdots,t^{\gamma_n}]]$ then $A$ is $(\Gamma, B)$-closed;
			
			\item [c)] If $R\in\{\bR,\bC\}$, $B=R\{t\}$\footnote{$R\{t\}$ denotes the ring of analytic functions on $t$ over $R$.} and  $A=R\{t^{\gamma_1},\hdots,t^{\gamma_n}\}$ then $A$ is $(\Gamma, B)$-closed.
		\end{enumerate}
	\end{proposition}
	
	\begin{proof}
		Let $f=\sum a_it^i\in B$. Thus, $f-f_{\mid_{G(\Gamma)}}=\sum\limits_{i\in\Gamma}a_it^i$.
		
		(a) Suppose that $B=R[t]$. In this case $f\in R[t]$, the set $D:=\{i\in\Gamma\mid a_i\neq 0\}$ is finite, and Lemma \ref{202405081704} ensures that $f-f_{\mid_{G(\Gamma)}}=\sum\limits_{i\in D}a_it^i\in A$.
		
		(b) Suppose that $A=R[[t^{\gamma_1},\hdots,t^{\gamma_n}]]$ and $B=R[[t]]$. For each $i\in\Gamma$ there exist $m_{ij}\in\bZ_{\geq 0}$ such that $i=\sum\limits_{j=1}^nm_{ij}\gamma_j$, and this implies $$f-f_{\mid_{G(\Gamma)}}=\sum\limits_{i\in\Gamma}a_it^i=\sum\limits_{i\in\Gamma}a_i(t^{\gamma_1})^{m_{i1}}\cdots(t^{\gamma_n})^{m_{in}}\in R[[t^{\gamma_1},\hdots,t^{\gamma_n}]]= A.$$
		
		(c) Completely analogous to the item (b).
	\end{proof}
	
	For concisely phrasing the cumbersome divisibility condition, we will use the following notation:

	\begin{definition}
		We say that the ring $R$ is $\gamma_1$-nice if one of the following conditions holds:
		
		\begin{enumerate}
			\item $R$ is a domain and $\car(R)\nmid \gamma_1$;
			
			\item $R$ is a reduced ring and $\car\left(\quotient{R}{\fp}\right)\nmid\gamma_1,\forall \fp\in\spec R$.
		\end{enumerate}
	\end{definition}

	\begin{remark}
		$\gamma_1$-nice rings are nothing special. The notion is introduced solely for convenience. Examples include:
		
		\begin{enumerate}
			\item Any domain $R$ (hence any field) with $\car R\nmid\gamma_1$
			
			\item Any reduced algebra $R$ over a field $k$, with $\car(k)\nmid \gamma_1$ and with injective canonical composition $k\rightarrow R\rightarrow \quotient{R}{\fp}$ for all $\fp\in\spec R$ (in particular, finitely generated polynomial rings and affine algebras over $k$).
			
			\item Any finite product of fields $R=k_1\times\cdots\times k_s$, where $\car(k_i)\nmid \gamma_1,\forall i\in\{1,\hdots,s\}$.  
		\end{enumerate}
	\end{remark}
	
	For each subset $S$ of $B$ we denote $R\cdot S$ and $A\cdot S$ as the $R$-submodule and $A$-submodule of $B$ generated by $S$, respectively.
	
	\begin{lemma}\label{202407191931}
		Let $m\in\{2,\hdots,n\}$ and suppose that $R$ is $\gamma_1$-nice ring. If $A$ is $(\Gamma, B)$-closed then $$A^*_{B,R}\sub A+R\cdot t^{\bigcup\limits_{j=1}^{m-1}L_j}+A\cdot t^{\tilde{L}(m)}.$$
	\end{lemma}
	
	\begin{proof}
		Let $f\in A^*_{B,R}$ and write $f=\sum\limits_{i\in\bZ_{\geq 0}} a_it^i, a_i\in R$. We can decompose 
		$$f=\sum\limits_{i\in\Gamma}a_it^i+\sum\limits_{i\in G(\Gamma)}a_it^i.$$
		Since $A$ is $(\Gamma, B)$-closed, $f-f_{\mid_{G(\Gamma)}}=\sum\limits_{i\in\Gamma}a_it^i\in A$. Thus, it suffices to show that $$f_{\mid_{G(\Gamma)}}=\sum\limits_{i\in G(\Gamma)}a_it^i\in R\cdot t^{\bigcup\limits_{j=1}^{m-1}L_j}+A\cdot t^{\tilde{L}(m)}.$$ We now decompose $f_{\mid_{G(\Gamma)}}$ even further:
		$$f_{\mid_{G(\Gamma)}} = f_0 + \dots + f_{m-1} + f_m$$
		
		%\begin{smallmatrix}
		%	i \in G(\Gamma) \\
		%	0 \leq i < \gamma_1} a_it^i
	%	\end{smallmatrix}

\noindent where $f_0 = \sum\limits_{\begin{smallmatrix}
		i \in G(\Gamma) \\
		0 \leq i < \gamma_1
\end{smallmatrix}} a_it^i$, $f_m = \sum\limits_{\begin{smallmatrix}
		i \in G(\Gamma) \\
		i>\gamma_m
\end{smallmatrix}} a_it^i$ and $f_s = \sum\limits_{\begin{smallmatrix}
		i \in G(\Gamma) \\
		\gamma_s \leq i < \gamma_{s+1}
\end{smallmatrix}} a_it^i$, for $ 1 \leq s \leq m-1$. Indeed, we know that $d_s \mid \gamma_1$ for any $1 \leq s < m$ and hence the preceding corollaries provide vanishing of certain terms in each $f_s$:

(1) If $R$ is a domain and $\car(R)\nmid \gamma_1$ then $\car(R)\nmid d_s$, thus Corollary \ref{202405091127} implies that $a_i=0$, $\forall i\in\{\gamma_s,\hdots,\gamma_{s+1}-1\}$ with $d_s\nmid i$. 

(2) If $R$ is a reduced ring and $\car\left(\quotient{R}{\fp}\right)\nmid\gamma_1,\forall \fp\in\spec R$ then $\car\left(\quotient{R}{\fp}\right)\nmid d_s,\forall \fp\in\spec R$ and  Corollary \ref{202408082140} ensures that $a_i=0$, $\forall i\in\{\gamma_s,\hdots,\gamma_{s+1}-1\}$ with $d_s\nmid i$. 

Hence $f_0 = 0$  and each of the other $f_s$ ($1\leq s\leq m-1)$ possesses non-zero coefficients only in the remaining gap-positions in the respective range:
$$f_s = \sum\limits_{\begin{smallmatrix}
		i\in G(\Gamma)\\
		\gamma_s<i<\gamma_{s+1}\\
		d_s\mid i
\end{smallmatrix}}a_it^i\in R\cdot t^{L_s}.$$

To see that $t^i \in A\cdot t^{\tilde{L}(m)}$, for all $i\in G(\Gamma)$ with $i>\gamma_m$, we only need to observe that $t^{\gamma_1},t^{\gamma_m} \in A$ and that $A\cdot  t^{\tilde{L}(m)}$ contains a gapless sequence of $\gamma_1$ powers of $t$ starting at $t^{\gamma_m}$.
\end{proof}

If we denote

\begin{itemize}
\item $\cL_j:=\{\ell\in L_j\mid \ell\leq \gamma_j+\gamma_1-1\}, \forall j\in\{1,\hdots,n-1\}$

\item $\cL(m):=\left(\bigcup\limits_{j=1}^{m-1}\cL_j\right)\cup\tilde{L}(m), \forall m\in\{2,\hdots,n\}$,
\end{itemize}

\noindent for a description of $A^*_{B,R}$ as an $A$-module (and $A$-algebra) we can drop further generators with the same argument as at the end of the last proof, obtaining the next lemma.

\begin{lemma}\label{202407131716}
Let $m\in\{2,\hdots,n\}$. Then, $t^{L_s}\sub A\cdot t^{\bigcup\limits_{j=1}^{m-1}\cL_j}+A\cdot t^{\tilde{L}(m)}, \forall s\in\{1,\hdots,m-1\}$. In particular, if $R$ is a $\gamma_1$-nice ring and $A$ is $(\Gamma, B)$-closed. $$A^*_{B,R}\sub A+A\cdot t^{\bigcup\limits_{j=1}^{m-1}\cL_j}+A\cdot t^{\tilde{L}(m)}\sub A[t^{\cL(m)}]$$
\end{lemma}

\begin{corollary}
Suppose that $A$ is $(\Gamma, B)$-closed. If $R$ is a domain with $\car(R)\neq 2$ and $\gamma_1=2$  then $$A^*_{B,R}=A.$$ \noindent i.e., $A$ is Lipschitz saturated.
\end{corollary}

\begin{proof}
Since $\gamma_1\nmid\gamma_2$ then $2\nmid\gamma_2$ and $\gcd(2,\gamma_2)=1$, i.e., $d_2=1$. By Lemma \ref{202407131716} we have $A^*_{B,R}\sub A[t^{\cL(2)}]=A[t^{\tilde{L}(2)}]$. Observe that $c(\langle 2,\gamma_2\rangle)=\gamma_2-1$, and this implies that $\ell\in\langle 2,\gamma_2\rangle\sub \Gamma, \forall \ell\geq \gamma_2-1$. In particular, $\tilde{L}(2)=\emptyset$. Therefore, $A^*_{B,R}\sub A[t^\emptyset]=A$.
\end{proof}

Now, we gather all the results above in order to obtain the main theorem of this work, where we describe $A^*_{B,R}$ in their main natural structures.

\begin{theorem}\label{202407191955}
Let $r\in\{2,\hdots,n\}$ such that $d_r=1$. If $A$ is $(\Gamma, B)$-closed and $R$ is a noetherian $\gamma_1$-nice ring then \begin{enumerate}
	\item [a)] $A^*_{B,R}=A+R\cdot t^{\bigcup\limits_{j=1}^{r-1}L_j}+A\cdot t^{\tilde{L}(r)}$ ($R$-module description);
	
	\item [b)] $A^*_{B,R}=A+A\cdot t^{\bigcup\limits_{j=1}^{r-1}\cL_j}+A\cdot t^{\tilde{L}(r)}$ ($A$-module description);
	
	\item [c)] $A^*_{B,R}=A[t^{\cL(r)}]=A[t^{L(r)}]$ ($A$-algebra description).
\end{enumerate}
\end{theorem}

\begin{proof}

By \ref{202407191931}, \ref{202407131716} and \ref{202408151302} we have $$A^*_{B,R}\sub A+R\cdot t^{\bigcup\limits_{j=1}^{r-1}L_j}+A\cdot t^{\tilde{L}(r)}\sub A+A\cdot t^{\bigcup\limits_{j=1}^{r-1}\cL_j}+A\cdot t^{\tilde{L}(r)}\sub A[t^{\cL(r)}]\sub A[t^{L(r)}]\sub A^*_{B,R}.$$
\end{proof}

Notice that $d_n=1$. However, it is desirable to work with the smallest $r\in\{2,\hdots,n\}$ such that $d_r=1$ to avoid unecessary gaps to describe $A^*_{B,R}$.

\appendix

\section{Examples where $R$ is a noetherian domain}

In the following examples, we assume that $R$ is a noetherian domain with $\car(R)\nmid\gamma_1$ and $A$ is $(\Gamma,B)$-closed. Keep the cases considered in the Proposition \ref{202411260046} in mind, i.e., \begin{itemize}
\item  If $B=R[t]$;

\item If $B=R[[t]]$ and $A=R[[t^{\gamma_1},\hdots,t^{\gamma_n}]]$;

\item $B=R\{t\}$ and $A=R\{t^{\gamma_1},\hdots,t^{\gamma_n}\}$, for $R\in\{\bR,\bC\}$.
\end{itemize}

\begin{example}
Suppose that $\Gamma=\langle6, 25\rangle$. Since we need to assume that $\car(R)\nmid\gamma_1$ then we assume that $\car(R)\notin\{2,3\}$. Since $d_2=1$, we have:

\begin{itemize}
	\item $G(\Gamma)=\left\{\begin{matrix}
		1, 2, 3, 4, 5, 7, 8, 9, 10, 11, 13, 14, 15, 16, 17, 19, 20, 21, 22, 23, 26, 27, 28, \\
		29, 32, 33, 34, 35, 38, 39, 40, 41, 44, 45, 46, 47, 51, 52, 53, 57, 58, 59, 63, \\
		64, 65, 69, 70, 71, 76, 77, 82, 83, 88, 89, 94, 95, 101, 107, 113, 119
	\end{matrix}\right\}$;
	
	\item $\tilde{L}(2)=\{\ell\in G(\Gamma)\mid 26\leq \ell \leq 30\}=\{26, 27, 28, 29\}$;
	
	\item $\cL(2)=\cL_1\cup\tilde{L}(2)=\emptyset\cup\tilde{L}(2)=\{26,27,28,29\}$;
	
	\item $A^*_{B,R}=A[t^{\cL(2)}]$.
\end{itemize}

Thus:  $$A^*_{B,R}=A[t^{26}, t^{27},t^{28}, t^{29}].$$
\end{example}

\begin{example}
Suppose that $\Gamma=\langle 9, 12, 22\rangle$. In this case, we have $d_2=3$ and $d_3=1$. Assume that $\car(R)\neq 3$. We have:

\begin{itemize}
	\item $G(\Gamma)=\left\{\begin{matrix}
		1, 2, 3, 4, 5, 6, 7, 8, 10, 11, 13, 14, 15, 16, 17, 19, \\
		20, 23, 25, 26, 28, 29, 32, 35, 37, 38, 41, 47, 50, 59
	\end{matrix}\right\}$;
	
	\item $L_2=\{\ell\in G(\Gamma)\mid 12<\ell<22\mbox{ and }3\mid \ell\}=\{15\}$;
	
	\item $\cL_2=\{\ell\in L_2\mid \ell\leq 20\}=\{15\}$;
	
	\item $\tilde{L}(3)=\{\ell\in G(\Gamma)\mid 23\leq \ell\leq 30\}=\{23, 25, 26, 28, 29\}$;
	
	\item $\cL(3)=\cL_2\cup\tilde{L}(3)=\{15, 23, 25, 26, 28, 29\}$;
	
	\item $A^*_{B,R}=A[t^{\cL(3)}]$.
\end{itemize}

Hence,  $$A^*_{B,R}=A[t^{15}, t^{23}, t^{25}, t^{26}, t^{28}, t^{29} ].$$

\end{example}

\begin{example}
Suppose that $\Gamma=\langle 18, 24, 39, 55\rangle$. Then, $d_2=6, d_3=3$ and $d_4=1$. Assume that $\car(R)\notin\{2,3\}$. We have: 

\begin{itemize}
	\item $G(\Gamma)=\left\{\begin{matrix}
		1, 2, 3, 4, 5, 6, 7, 8, 9, 10, 11, 12, 13, 14, 15, 16, 17, 19, 20, 21, 22, 23, 25, 26, 27, \\
		28,	29, 30, 31, 32, 33, 34, 35, 37, 38, 40, 41, 43, 44, 45, 46, 47, 49, 50, 51, 52, 53, \\
		56, 58, 59, 61, 62, 64, 65, 67, 68, 69, 70, 71, 74, 76, 77, 80, 82, 83, 85, 86, 88, 89, \\
		92, 95, 98, 100, 101, 104, 106, 107, 113, 116, 119, 122, 124, 125, 131, 137, 140, \\
		143, 155, 161, 179
	\end{matrix}\right\}$;
	
	\item $L_2=\{\ell\in G(\Gamma)\mid 24<\ell<39\mbox{ and }6\mid \ell\}=\{30\}$;
	
	\item $\cL_2=\{\ell\in L_2\mid \ell\leq 41\}=\{30\}$;
	
	\item $L_3=\{\ell\in G(\Gamma)\mid 39<\ell<55\mbox{ and }3\mid \ell\}=\{45, 51\}$;
	
	\item $\cL_3=\{\ell\in L_3\mid \ell\leq 56\}=\{45, 51\}$;
	
	\item $\tilde{L}(4)=\{\ell\in G(\Gamma)\mid 56\leq \ell\leq 72\}=\{56, 58, 59, 61, 62, 64, 65, 67, 68, 69, 70, 71\}$;
	
	\item $\cL(4)=\{30, 45, 51, 56, 58, 59, 61, 62, 64, 65, 67, 68, 69, 70, 71\}$;
	
	\item $A^*_{B,R}=A[t^{\cL(4)}]$.
\end{itemize}

Therefore,  after calculating a minimal generator set of the semigroup generated by $\{18, 24\}\cup\cL(4)$ we conclude that $$A^*_{B,R}=A[t^{30}, t^{45}, t^{51}, t^{56}, t^{58}, t^{59}, t^{61}, t^{62}, t^{64}, t^{65}, t^{67}, t^{68}, t^{70}, t^{71}].$$

\end{example}

\begin{example}
Suppose that $\Gamma=\langle 40, 60, 70, 85, 103 \rangle$. Then, $d_2=20, d_3=10, d_4=5$ and $d_5=1$. Assume that $\car(R)\notin\{2,5\}$. We have:

\begin{itemize}
	\item $G(\Gamma)=\left\{\begin{matrix}
		1, 2, 3, 4, 5, 6, 7, 8, 9, 10, 11, 12, 13, 14, 15, 16, 17, 18, 19, 	20, 21, 22, 23, 24,	25, 26, 27, \\
		28, 29, 30, 31, 32, 33, 34, 35, 36, 37, 38, 39, 41, 42, 43, 44, 45, 46, 47, 48, 49, 50, 51, 52, \\
		53, 54, 55, 56, 57, 58, 59, 61, 62, 63, 64, 65, 66, 67, 68, 69, 71, 72, 73, 74, 75, 76, 77, 78, \\
		79, 81, 82, 83, 84, 86, 87, 88, 89, 90, 91, 92, 93, 94, 95, 96, 97, 98, 99, 101, 102, 104, 105,\\
		106, 107, 108, 109, 111, 112, 113, 114, 115, 116, 117, 118, 119, 121, 122, 123, 124, 126, \\
		127, 128, 129, 131, 132, 133, 134, 135, 136, 137, 138, 139, 141, 142, 144, 146, 147, 148, \\
		149, 151, 152, 153, 154, 156, 157, 158, 159, 161, 162, 164, 166, 167, 168, 169, 171, 172, \\
		174, 175, 176, 177, 178, 179, 181, 182, 184, 186, 187, 189, 191, 192, 193, 194, 196, 197, \\
		198, 199, 201, 202, 204, 207, 208, 209, 211, 212, 214, 216, 217, 218, 219, 221, 222, 224,\\
		226, 227, 229, 231, 232, 234, 236, 237, 238, 239, 241, 242, 244, 247, 249, 251, 252, 254, \\
		256, 257, 259, 261, 262, 264, 267, 269, 271, 272, 274, 277, 278, 279, 281, 282, 284, 287, \\
		289, 292, 294, 296, 297, 299, 301, 302, 304, 307, 311, 312, 314, 317, 319, 321, 322, 324, \\
		327, 329, 332, 334, 337, 339, 341, 342, 344, 347, 352, 354, 357, 359, 362, 364, 367, 372, \\
		374, 377, 381, 382, 384, 387, 392, 397, 399, 402, 404, 407, 414, 417, 422, 424, 427, 432, \\
		437, 442, 444, 447, 457, 462, 467, 477, 484, 487, 502, 507, 517, 527, 547, 587
	\end{matrix}\right\}$;
	
	\item $L_2=\{\ell\in G(\Gamma)\mid 60<\ell<70\mbox{ and }20\mid\ell\}=\emptyset$ and $\cL_2=\emptyset$;
	
	\item $L_3=\{\ell\in G(\Gamma)\mid 70<\ell<85\mbox{ and }10\mid \ell\}=\emptyset$ and $\cL_3=\emptyset$;
	
	\item $L_4=\{\ell\in G(\Gamma)\mid 85<\ell<103\mbox{ and }5\mid\ell\}=\{90, 95\}$;
	
	\item $\cL_4=\{\ell\in L_4\mid \ell\leq 124\}=\{90, 95\}$;
	
	\item $\tilde{L}(5)=\{\ell\in G(\Gamma)\mid 104\leq\ell\leq 142\}=\left\{\begin{matrix}
		104, 105, 106, 107, 108, 109, 111, 112, 113, \\
		114, 115, 116, 117, 118, 119, 121, 122, 123, \\
		124, 126, 127, 128, 129, 131, 132, 133, 134, \\
		135, 136, 137, 138, 139, 141, 142
	\end{matrix}\right\}$;
	
	\item $\cL(5)=\cL_4\cup\tilde{L}(5)=\left\{\begin{matrix}
		90, 95, 104, 105, 106, 107, 108, 109, 111, 112, 113, 114, 115, 116, \\
		117, 118, 119, 121, 122, 123, 124, 126, 127, 128, 129, 131, 132, \\
		133, 134, 135, 136, 137, 138, 139, 141, 142
	\end{matrix}\right\}$;
	
	\item $A^*_{B,R}=A[t^{\cL(5)}]$.
\end{itemize}

Therefore, after calculating a minimal generator set for the semigroup generated by $\{40, 60, 70, 85\}\cup\cL(5)$, we conclude that $$A^*_{B,R}=A\left[\begin{matrix}
	t^{90}, t^{95}, t^{104}, t^{105}, t^{106}, t^{107}, t^{108}, t^{109}, t^{111}, \\
	t^{112}, t^{113}, t^{114}, t^{115}, t^{116}, t^{117}, t^{118}, t^{119}, t^{121}, \\
	t^{122}, t^{123}, t^{124}, t^{126}, t^{127}, t^{128}, t^{129}, t^{131}, t^{132}, \\
	t^{133}, t^{134}, t^{136}, t^{137}, t^{138}, t^{139}, t^{141}, t^{142}
\end{matrix}\right].$$
\end{example}

\begin{example}
Suppose that $\Gamma=\langle 12, 18, 22, 29, 35, 49 \rangle$. Then, $d_2=6$, $d_3=2$ and $d_4=1$. Assume that $\car(R)\notin\{2,3\}$. We have:

\begin{itemize}
	\item $G(\Gamma)=\left\{\begin{matrix}
		1, 2, 3, 4, 5, 6, 7, 8, 9, 10, 11, 13, 14, 15, 16, 17, 19, 20, 21,\\ 
		23, 25, 26, 27, 28, 31, 32, 33, 37, 38, 39, 43, 45, 50, 55
	\end{matrix}\right\}$;
	
	\item $L_2=\{\ell\in G(\Gamma)\mid 18<\ell<22\mbox{ and }6\mid\ell\}=\emptyset$ and $\cL_2=\emptyset$;
	
	\item $L_3=\{\ell\in G(\Gamma)\mid 22<\ell<29\mbox{ and }2\mid\ell\}=\{26, 28\}$;
	
	\item $\cL_3=\{\ell\in L_3\mid \ell\leq 33\}=\{26, 28\}$;
	
	\item $\tilde{L}(4)=\{\ell\in G(\Gamma)\mid 30\leq\ell\leq 40\}=\{31, 32, 33, 37, 38, 39\}$;
	
	\item $\cL(4)=\cL_2\cup\cL_3\cup\tilde{L}(4)=\{26, 28, 31, 32, 33, 37, 38, 39\}$;
	
	\item $A^*_{B,R}=A[t^{\cL(4)}]$.
\end{itemize}

Thus, after calculating a minimal generator set for the semigroup generated by $\{12, 18, 22\}\cup\cL(4)$, we conclude that $$A^*_{B,R}=A[t^{26}, t^{28}, t^{31}, t^{32}, t^{33}, t^{37}, t^{39}].$$

\end{example}

\section*{Acknowledgements}

\hspace{0.47cm} Thiago da Silva is funded by CAPES grant number 88887.909401/2023-00 and CAPES grant number 88887.897201/2023-00. Anne Fr\"uhbis-Kr\"uger is partially supported by the German Research Foundation (DFG) through TRR
195, Project II.5. 

The authors are grateful to Prof. Antônio Carlos Telau (UFVJM-Brazil) for providing a program generating the gaps of numerical semigroups, which led to the observations initiating this work.

\vspace{0.5cm}

\textsc{Anne Frühbis-Krüger (Carl von Ossietzky Universität Oldenburg)}

anne.fruehbis-krueger@uni-oldenburg.de

\vspace{0.5cm}

\textsc{Thiago da Silva (Federal University of Espírito Santo)}

thiago.silva@ufes.br

\end{document}